\documentclass[11pt,reqno]{amsart}

\usepackage{xspace}
\usepackage{tikz}
\usepackage{amsmath,amsthm,amsfonts,amssymb,latexsym,mathrsfs,color}
\usepackage{hyperref}

\newtheorem{theorem}{Theorem}

\theoremstyle{definition}
\newtheorem{definition}[theorem]{Definition}
\newtheorem{ex}[theorem]{Example}

\newcommand{\op}[1]{\operatorname{#1}}

\newcommand{\ASC}{\op{ASC}}
\newcommand{\AP}{\op{AP}}
\newcommand{\MK}{\op{MARK}}
\newcommand{\ODD}{\op{ODD}}
\newcommand{\EVEN}{\op{EVEN}}
\newcommand{\LRMIN}{\op{LRMIN}}
\newcommand{\RLMIN}{\op{RLMIN}}

\newcommand{\asc}{\op{asc}}
\newcommand{\ap}{\op{ap}}
\newcommand{\mk}{\op{mark}}
\newcommand{\odd}{\op{odd}}
\newcommand{\even}{\op{even}}
\newcommand{\lrmin}{\op{lrmin}}
\newcommand{\rlmin}{\op{rlmin}}
\newcommand{\des}{\op{des}}
\newcommand{\desrlmin}{\op{desrlmin}}
\newcommand{\caplat}{\op{cap}}
\newcommand{\cyc}{\op{cyc}}
\newcommand{\exc}{\op{exc}}
\newcommand{\fix}{\op{fix}}
\newcommand{\bk}{\op{bk}}
\newcommand{\redd}{\op{red}}

\newcommand{\stati}{\op{stat}_i}
\newcommand{\Gen}{\op{Gen}}
\newcommand{\C}{\mathcal{C}}
\newcommand{\m}{{\rm M}}

\newcommand{\ms}{\mathfrak{S}}
\newcommand{\msn}{\mathfrak{S}_n}
\newcommand{\mq}{\mathcal{Q}}
\newcommand{\mqn}{\mathcal{Q}_n}
\newcommand{\mdq}{\mathcal{DQ}}

\newcommand{\mmn}{\mathcal{M}_{2n}}

\newcommand{\stirling}[2]{\genfrac{[}{]}{0pt}{}{#1}{#2}}
\newcommand{\arxiv}[1]{\href{http://arxiv.org/abs/#1}{\texttt{arXiv:#1}}}
\title{Stirling permutations, marked permutations and Stirling derangements}
\author[G.-H.~Duh]{Guan-Huei Duh}
\address{Institute of Mathematics,
        Academia Sinica, Taipei, Taiwan}
\email{arthurduh1@gmail.com (G.-H.~Duh)}
\author[Y.-C.~Lin]{Yen-Chi Roger Lin}
\address{Department of Mathematics,
        National Taiwan Normal University,
         Taipei 116, Taiwan}
\email{yclinpa@gmail.com (Y.-C. Lin)}
\author[S.-M.~Ma]{Shi-Mei Ma}
\address{School of Mathematics and Statistics,
        Northeastern University at Qinhuangdao,
         Hebei 066004, P.R. China}
\email{shimeimapapers@163.com (S.-M. Ma)}
\author[Y.-N. Yeh]{Yeong-Nan Yeh}
\address{Institute of Mathematics,
        Academia Sinica, Taipei, Taiwan}
\email{mayeh@math.sinica.edu.tw (Y.-N. Yeh)}

\begin{document}
\begin{abstract}
In this paper we introduce the definition of marked permutations.
We first present a bijection between Stirling permutations and marked permutations.
We then present an involution on Stirling derangements. Furthermore, we present
a symmetric bivariate enumerative polynomials on $r$-colored marked permutations.
Finally, we give an explanation of $r$-colored marked permutations by using the language of combinatorial objects.
\end{abstract}


\maketitle

\section{Introduction}
Let $\msn$ denote the symmetric group of all permutations of $[n]$, where $[n]=\{1,2,\ldots,n\}$.
Let $\pi=\pi_1\pi_2\cdots \pi_n\in\msn$.
An \emph{ascent} of $\pi$ is an entry $\pi_i$, $i\in \{2,3,\dots,n\}$, such that $\pi_i>\pi_{i-1}$.
Denote by $\ASC(\pi)$ the set of all ascents of $\pi$, and $\asc(\pi)=\left|\ASC(\pi)\right|$.
For example, $\asc(53\textbf{4}1\textbf{2})=\left|\{2,4\}\right|=2$.
The classical Eulerian polynomial is defined by
\begin{equation*}
A_n(x)=\sum_{\pi\in\msn}x^{\asc(\pi)}.
\end{equation*}
Set $A_0(x)=1$.
The exponential generating function for $A_n(x)$ is
\begin{equation}\label{Axz}
A(x,t)=\sum_{n\geq 0}A_n(x)\frac{t^n}{n!}=\frac{1-x}{e^{t(x-1)}-x}.
\end{equation}

Stirling permutations were introduced by Gessel and Stanley~\cite{Gessel78}. Let $[n]_2$ denote the multiset $\{1,1,2,2,\dots,n,n\}$.
A \emph{Stirling permutation} of order $n$ is a permutation of $[n]_2$ such that
every entry between the two occurrences of $i$ is greater than $i$ for each $i\in [n]$.
Various statistics on Stirling permutations have been extensively studied in the past decades, including descents~\cite{Chen16,Gessel78,Haglund12}, plateaux~\cite{Bona08,Chen16,Haglund12}, blocks~\cite{Remmel14}, ascent
plateaux~\cite{Ma14,Ma1601} and cycle ascent plateaux~\cite{Ma1602}.

Denote by $\mqn$ the set of Stirling permutations of order $n$ and let $\sigma=\sigma_1\sigma_2\cdots\sigma_{2n}\in\mqn$. An occurrence of
an \emph{ascent plateau} is an entry $\sigma_i$, $i\in \{2,3,\ldots,2n-1\}$, such that
$\sigma_{i-1}<\sigma_i=\sigma_{i+1}$ (see~\cite{Ma14}).
Let $\AP(\sigma)$ be the set of all ascent plateaux of $\sigma$, and
$\ap(\sigma)=\left|\AP(\sigma)\right|$.
As an example, $\ap(2211\textbf{3}3)=\left|\{3\}\right|=1$.
Let $$N_n(x)=\sum_{\pi\in\mqn}x^{\ap(\pi)}.$$
Set $N_0(x)=1$.
From~\cite[Theorem 2]{Ma14}, we have
\begin{equation}\label{Ax2t}
\sum_{n\geq 0}N_n(x)\frac{t^n}{n!}=\sqrt{A(x,2t)}.
\end{equation}

Let $w=w_1w_2\cdots w_n$ be a word on $[n]$. A \emph{left-to-right minimum}
of $w$ is an element $w_i$ such that $w_i<w_j$ for every $j\in [i-1]$ or $i=1$;
a \emph{right-to-left minimum}
of $w$ is an element $w_i$ such that $w_i<w_j$ for every $j\in \{i+1,i+2,\ldots,n\}$ or $i=n$.
Let $\LRMIN(w)$ and $\RLMIN(w)$ denote the set of entries of left-to-right minima and right-to-left minima of $w$, respectively.
Set $\lrmin(w)=\left|\LRMIN(w)\right|$ and $\rlmin(w)=\left|\RLMIN(w)\right|$.
For example, $\lrmin(\textbf{2}233\textbf{1}1)=\left|\{1,2\}\right|=2$ and
$\rlmin(22331\textbf{1})=\left|\{1\}\right|=1$.

Motivated by~\eqref{Ax2t}, we now introduce the definition of marked permutations.
Given a permutation $\pi=\pi_1\pi_2\cdots \pi_n\in\msn$,
a \emph{marked permutation} is a permutation with marks on some of its non-left-to-right minima.
An element $i$ is denoted by $\overline{\imath}$ when it is marked. Let $\overline{\ms}_n$ denote the set of marked permutations of $[n]$.
For example, $\overline{\ms}_1=\{1\},~\overline{\ms}_2=\{12,1\overline{2},21\}$ and
\begin{align*}
\overline{\ms}_3 =&\{1~2~3,\ 1~3~2,\ 2~1~3,\ 2~3~1,\ 3~1~2,\ 3~2~1,\ 1~\overline{2}~3,\ 1~2~\overline{3}, \\
&1~\overline{2}~\overline{3},\ 1~\overline{3}~2,\ 1~3~\overline{2},\ 1~\overline{3}~\overline{2},\
2~1~\overline{3},\ 2~\overline{3}~1,\ 3~1~\overline{2}\}.
\end{align*}

This paper is organized as follows.
In Section~\ref{Section02}, we present a bijection between Stirling permutations and marked permutations.
In Section~\ref{Section03}, we use an involution to prove an identity between Stirling derangements and perfect matchings.
In Section~\ref{Section04}, we present
a symmetric bivariate enumerative polynomials on $r$-colored marked permutations.
And in Section~\ref{Section05}, we use the language of combinatorial objects to give an explanation of marked permutations.

\section{A bijection between Stirling permutations and marked permutations}
\label{Section02}
Let $\sigma=\sigma_1\sigma_2\cdots\sigma_{2n}\in\mqn$ be a Stirling permutation of order $n$.
An entry $k$ of $\sigma$ is called an \emph{even indexed entry} (resp.~\emph{odd indexed entry}) if the first appearance of $k$ occurs at an even (resp.~odd) position of $\sigma$.
Let $\EVEN(\sigma)$ (resp.~$\ODD(\sigma)$) denote the set of even (resp.~odd) indexed entries of $\sigma$,
$\even(\sigma)=\left|\EVEN(\sigma)\right|$, and $\odd(\sigma)=\left|\ODD(\sigma)\right|$.
For example, $\even(221\textbf{3}31)=\left|\{3\}\right|=1$ and $\odd(\textbf{2}2\textbf{1}{3}31)=\left|\{1,2\}\right|=2$.

Given $\pi\in \overline{\ms}_n$,
let $\MK(\pi)$ denote the set of marked entries of $\pi$, and $\mk(\pi)=\left|\MK(\pi)\right|$.
The statistics $\ASC(\pi), \asc(\pi), \LRMIN(\pi)$ and $\lrmin(\pi)$ are defined by forgetting the marks of $\pi$.
For example, $\asc(3\overline{\textbf{5}}21\overline{\textbf{4}67})=4$, $\mk(3\overline{\textbf{5}}21\overline{\textbf{4}}67)=2$
and $\lrmin(\textbf{3}\overline{5}\textbf{2}\textbf{1}\overline{4}67)=3$.

A \emph{block} in an element of $\mqn$ or $\overline{\ms}_n$ is a substring which begins with a left-to-right minimum,
and contains exactly this one left-to-right minimum;
moreover, the substring is maximal, i.e. not contained in any larger such substring.
It is easily derived by induction that any Stirling
permutation or permutation has a unique decomposition as a sequence of blocks.
\begin{ex}
The block decompositions of $34664325527711$ and $3\overline{4}62\overline{5}71$ are respectively given by $[346643][255277][11]$ and $[3\overline{4}6][25\overline{7}][1]$.
\end{ex}

We now present the first result of this paper.
\begin{theorem}\label{thm01}
For $n\geq 1$, we have
\begin{equation}
\sum_{\sigma\in\mqn} q^{\lrmin(\sigma)} x^{\ap(\sigma)} y^{\even(\sigma)}=
\sum_{\pi\in \overline{\ms}_n}q^{\lrmin(\pi)} x^{\asc(\pi)} y^{\mk(\pi)}.
\end{equation}
\end{theorem}
\begin{proof}
We prove a stronger result: there is a one-to-one correspondence between the set statistics $(\LRMIN,\AP,\EVEN)$ on Stirling permutations and $(\LRMIN,\ASC,\MK)$ on marked permutations.

Define
\begin{align*}
Q(n;X,Y,Z)&=\{\sigma\in\mqn\mid \LRMIN(\sigma)=X,\AP(\sigma)=Y,\EVEN(\sigma)=Z\}, \mbox{ and} \\
S(n;X,Y,Z)&=\{\pi\in \overline{\ms}_n\mid \LRMIN(\pi)=X,\ASC(\pi)=Y,\MK(\pi)=Z\},
\end{align*}
where $X,Y,Z\subseteq[n]$.
Clearly, they form partitions of all Stirling permutations and marked permutations respectively.

Now we start to construct a bijection, denoted by $\Phi$, between Stirling permutations and marked permutations.
In addition, it maps each $Q(n;X,Y,Z)$ onto $S(n;X,Y,Z)$.
When $n=1$, it is clear that $11\in Q(1;\{1\},\emptyset,\emptyset)$ and $1\in S(1;\{1\},\emptyset,\emptyset)$.
Hence setting $\Phi(11)=1$ satisfies the requirement.

Fix $m\geq2$, and assume that $\Phi$ is a bijection between $Q(m-1;X,Y,Z)$ and $S(m-1;X,Y,Z)$ for all possible $X,Y,Z$.
Let $\sigma'\in\mq_m$ be obtained from some $\sigma\in Q(m-1;X,Y,Z)$
by inserting the substring $mm$ into $\sigma$.
By the assumption, $\Phi(\sigma)=\pi\in S(m-1;X,Y,Z)$.

If $mm$ is placed at the front of $\sigma$, that is,
  $\sigma'= mm\sigma$, then we let $\Phi(\sigma')=m\pi$. In this case, we have
  $\sigma'\in Q(m;X\cup\{m\},Y,Z)$ and $\Phi(\sigma')\in S(m;X\cup\{m\},Y,Z)$.
Notice that this is the only way to produce a new block.

Otherwise, suppose $\sigma'$ is obtained from $\sigma$
by inserting $mm$ into the $p^{\text{th}}$ block.
Let $r$ be the left-to-right minimum contained in the $p^{\text{th}}$ block of $\sigma$.
There are three possible cases:
\begin{enumerate}
  \item [\rm (i)]
  If $mm$ is inserted immediately before the second $r$, then $\Phi(\sigma')=\pi'$ is obtained by inserting
  a marked $\overline{m}$ at the end of the $p^{\text{th}}$ block of $\pi$.
  Note that $m$ is an additional even
indexed entry, as well an ascent plateau, of $\sigma'$ after inserting $mm$ into $\sigma$.
Meanwhile, $m$ is a marked element, as well an ascent, of $\pi'$.
Hence $\sigma'\in Q(m;X,Y\cup\{m\},Z\cup\{m\})$ and $\Phi(\sigma')\in S(m;X,Y\cup\{m\},Z\cup\{m\})$.
  \item [\rm (ii)] If $mm$ is inserted immediately before $s$, $s\neq r$, then $\Phi(\sigma')=\pi'$ is obtained by inserting
  $m$ or $\overline{m}$ into the $p^{\text{th}}$ block of $\pi$ such that $m$ is immediately before $s$.
  The inserted entry $m$ is marked if and only if $m$ is an even
  indexed entry of $\sigma'$.
  When $s\in Y$, let $Y'=\left(Y\cup\{m\}\right)\backslash\{s\}$. Otherwise, let $Y'=Y\cup\{m\}$.
  When $m$ is an even indexed entry, let $Z'=Z\cup\{m\}$. Otherwise, let $Z'=Z$.
  In each possible case, we see that $\sigma'\in Q(m;X,Y',Z')$ and $\Phi(\sigma')\in S(m;X,Y',Z')$.
 \item [\rm (iii)] If $mm$ is inserted at the end of the $p^{\text{th}}$ block, then $\Phi(\sigma')=\pi'$ is obtained by inserting
  an unmarked $m$
  at the end of the $p^{\text{th}}$ block of $\pi$.
  Note that $\sigma$ does not gain any additional even indexed entry after inserting $m$, but obtain the ascent plateau
  of $m$. On the other hand, $m$ is a new ascent of $\pi'$ after inserting $m$ into $\pi$.
  Hence $\sigma'\in Q(m;X,Y\cup\{m\},Z)$ and $\Phi(\sigma')\in S(m;X,Y\cup\{m\},Z)$.
\end{enumerate}
The above argument shows that $\Phi(\mqn)\subseteq \overline{\ms}_n$, and that $\Phi$ is one-to-one on $\mqn$.
Since the cardinality of $\mqn$ is the same as that of $\overline{\ms}_n$, $\Phi$ must be a bijection between
$\mqn$ and $\overline{\ms}_n$.
By induction, we see that $\Phi$ is the desired
bijection between Stirling permutations and marked permutations.
\end{proof}
\begin{ex}
Consider $\sigma=266255133441\in \mq_6$. The correspondence between $\sigma$ and
$\Phi(\sigma)$ is built
up as follows:
\begin{equation*}
  \begin{array}{rcl}
    [11] & \Leftrightarrow & [1] \\ \relax
    [22][11] & \Leftrightarrow& [2] \, [1] \\ \relax
    [22][1331] & \Leftrightarrow& [2] \, [1 \, \overline{3}] \\
    \relax
    [22][134431] &\Leftrightarrow & [2] \, [1 \, {4} \, \overline{3}] \\ \relax
    [2255][134431] & \Leftrightarrow& [2 \, 5] \, [1 \, {4} \, \overline{3}] \\ \relax
    [266255][134431] & \Leftrightarrow& [2 \, 5 \, \overline{6}] \, [1 \, {4} \, \overline{3}]
  \end{array}
\end{equation*}
\end{ex}

Exploiting the bijection $\Phi$ used in the proof of Theorem~\ref{thm01}, we also get the following result.
\begin{theorem}
  For $n \geq 1$, we have
  \begin{equation}
    \label{eq:sign-balance-mark}
    \sum_{\sigma\in\mqn} x^{\ap(\sigma)}(-1)^{\even(\sigma)}
    = \sum_{\pi\in \overline{\ms}_n}x^{\asc(\pi)}(-1)^{\mk(\pi)}
    = 1.
  \end{equation}
\end{theorem}
\begin{proof}
From Theorem~\ref{thm01}, we get the first equality of~\eqref{eq:sign-balance-mark}.
For the last equality of~\eqref{eq:sign-balance-mark}, we will consider
an involution $\iota$ on $\overline{\ms}_n$ as follows.  For a marked permutation $\pi$ that has non-left-to-right minima, define $\iota(\pi)$ to be the marked permutation obtained from $\pi$ by changing the marking of the smallest non-left-to-right minimum of $\pi$. For example, when $\pi = 2 \, \overline{\textbf{5}} \, \textbf{3} \, 1 \, \textbf{4}$, then $\iota(\pi) = 2 \, \overline{\textbf{5}} \, \overline{\textbf{3}} \, 1 \, \textbf{4}$. This map $\iota$ is clearly a sign-reversing involution because the number of marks is either plus 1 or minus 1. It is also clear that $\iota$ preserves the ascent statistic since no entry leaves its position.  The only marked permutation in $\overline{\ms}_n$ which cannot be mapped by $\iota$ is the one in which every entry is a left-to-right minimum, i.e., the marked permutation $n \, (n-1) \, \cdots \, 2 \, 1$, whose ascent is $0$. This completes the proof.
\end{proof}

\begin{theorem}\label{thm02}
We have
  \begin{equation*}
   \sum_{n\geq0} \sum_{\pi\in \overline{\ms}_n} q^{\lrmin(\pi)}x^{\asc(\pi)}y^{\mk(\pi)} \frac{t^n}{n!}
    = A\bigl(x,t(1+y)\bigr)^{\frac{q}{1+y}},
  \end{equation*}
  where $A(x,z)$ is the exponential generating function given by~\eqref{Axz}.
\end{theorem}
\begin{proof}
Combining~\cite[Proposition 7.3]{Bre00} and the fundamental transformation introduced by Foata and Sch\"utzen\-berger~\cite{FS70}, we have
\begin{equation*}
\sum_{n\geq0} \sum_{\pi\in\msn} q^{\lrmin(\pi)} x^{\asc(\pi)} \frac{t^n}{n!}
=\sum_{n\geq0}\sum_{\pi\in\msn} q^{\cyc(\pi)} x^{\exc(\pi)} \frac{t^n}{n!}
=A(x,t)^q
\end{equation*}

For a permutation $\pi$ in $\msn$ with $\lrmin(\pi)=\ell$, there are $n-\ell$ entries that could be either marked or not. Therefore, we have
\begin{align*}
\sum_{n\geq0} \sum_{\pi\in \overline{\ms}_n} q^{\lrmin(\pi)}x^{\asc(\pi)}y^{\mk(\pi)} \frac{t^n}{n!}&=
\sum_{n\geq0} \sum_{\pi\in\msn} q^{\lrmin(\pi)} x^{\asc(\pi)} (1+y)^{n-\lrmin(\pi)} \frac{t^n}{n!} \\
&= \sum_{n\geq0} \sum_{\pi\in\msn} \left(\frac{q}{1+y}\right)^{\lrmin(\pi)} x^{\asc(\pi)} \frac{(t(1+y))^n}{n!} \\
&= A\bigl(x,t(1+y)\bigr)^{\frac{q}{1+y}}.
\end{align*}
\end{proof}

Combining Theorem~\ref{thm01} and Theorem~\ref{thm02}, we have
\begin{equation*}
\sum_{n\geq0}\sum_{\sigma\in\mqn}y^{\even(\sigma)} \frac{t^n}{n!}
=\sum_{n\geq0}\sum_{\pi\in \overline{\ms}_n} y^{\mk(\pi)} \frac{t^n}{n!}
=A\bigl(1,t(1+y)\bigr)^{\frac{1}{1+y}}.
\end{equation*}

Let $\stirling{n}{i}$ be the (signless) Stirling number of the first kind, i.e., the number of permutations of $\msn$ with $i$ cycles.
Note that $\even(\sigma)+\odd(\sigma)=n$ for $\sigma \in\mqn$. Let
\begin{equation*}
  E_n(p,q)=\sum_{\sigma\in\mqn}p^{\odd(\sigma)}q^{\even(\sigma)}.
\end{equation*}
Now we present the following result.
\begin{theorem}
For $n\geq 1$, we have
\begin{equation}\label{Enpq}
E_n(p,q)=\sum_{k=0}^n\stirling{n}{k}p^k(p+q)^{n-k}.
\end{equation}
In particular,
$E_n(1,1)=(2n-1)!!,~E_n(p,0)=n!p^n,E_n(1,-1)=1,~E_n(-1,1)=(-1)^n$
for $n\geq 2$.
\end{theorem}
\begin{proof}
There are two ways in which a permutation $\sigma'\in\mq_{n}$ with $\even(\sigma')=k$ can be obtained from a permutation $\sigma=\sigma_1\sigma_2\cdots \sigma_{2n-2}\in\mq_{n-1}$.
\begin{enumerate}
  \item [(i)]
  If $\even(\sigma)=k-1$, then we can insert the two copies of $n$ right after $\sigma_{2i-1}$, where $i\in [n-1]$.
  \item [(ii)] If $\even(\sigma)=k$, then we can insert the two copies of $n$ in the front of $\sigma$ or right after $\sigma_{2i}$, where $i\in [n-1]$.
\end{enumerate}
Clearly, $E(1,0)=1$ corresponds to $11\in\mq_1$. Therefore, the numbers $E(n,k)$ satisfy the
recurrence relation
\begin{equation}\label{Enk-recu}
E(n,k)=(n-1)E(n-1,k-1)+nE(n-1,k),
\end{equation}
with the initial conditions $E(1,0)=1$ and $E(1,k)=0$ for $k\geq 1$.
It follows from~\eqref{Enk-recu}
that
\begin{equation*}
E_n(q)=(n+(n-1)q)E_{n-1}(q)
\end{equation*}
for $n\geq 1$, with the initial value $E_0(q)=1$.
Thus
$$E_n(q)=\prod_{i=1}^n(i+(i-1)q).$$
Recall that
$$\sum_{k=0}^n\stirling{n}{k}x^k=x(x+1)\cdots (x+n-1).$$
Therefore,
\begin{equation*}\label{Enq-explicit}
E_n(q)=\sum_{k=0}^n\stirling{n}{k}(1+q)^{n-k},
\end{equation*}
which leads to~\eqref{Enpq}, since $$E_n(p,q)=p^nE_n\left(\frac{q}{p}\right).$$
\end{proof}

\section{Stirling derangements and increasing binary trees}
\label{Section03}
Let $[k]^n$ denote the set of words of length $n$ in the alphabet $[k]$. For $\omega=\omega_1\omega_2\cdots\omega_n\in [k]^n$,
the reduction of $\omega$, denoted by $\redd(\omega)$, is the unique word of length $n$ obtained by replacing
the $i$th smallest entry by $i$. For example, $\redd(33224547)=22113435$.

Very recently, Ma and Yeh~\cite{Ma1602} introduced the definition of Stirling permutations of the second kind.
A permutation $\sigma$ of the multiset $[n]_2$ is a \emph{Stirling permutation of the second kind} of order $n$ whenever $\sigma$ can be written as a nonempty disjoint union of its distinct cycles and $\sigma$ has a standard cycle form that satisfies the following conditions:
\begin{itemize}
  \item [\rm (i)] For each $i\in [n]$, the two copies of $i$ appear in exactly one cycle;
  \item [\rm (ii)] Each cycle is written with one of its smallest entry first and the cycles are written in increasing order of their smallest entries;
  \item [\rm (iii)] The reduction of the word formed by all entries of each cycle is a Stirling permutation. In other words, if $(c_1,c_2,\ldots,c_{2k})$ is a cycle of $\sigma$, then $\redd(c_1c_2\cdots c_{2k})\in \mq_k$.
\end{itemize}
Let $\mqn^2$ denote the set of Stirling permutations of the second kind of order $n$.
In the following discussion, we always write $\sigma\in\mqn^2$ in standard cycle form.
Let $(c_1,c_2,\ldots,c_{2k})$ be a cycle of $\sigma$, where $k\geq2$.
An entry $c_i$ is called a \emph{cycle ascent plateau} if
$c_{i-1}<c_{i}=c_{i+1}$, where $2\leq i\leq2k-1$.
Denote by $\caplat(\sigma)$ (resp.~$\cyc(\sigma)$) the number of cycle ascent plateaux (resp.~cycles) of $\sigma$.
An entry $k\in[n]$ be called a \emph{fixed point} of $\sigma$ if $(kk)$ is a cycle of $\sigma$.
Let $\fix(\sigma)$ denote the number of fixed points of $\sigma$.
Using the fundamental transformation of Foata and Sch\"utzenberger~\cite{FS70}, we have
\begin{equation}\label{mqn2}
  \sum_{\sigma\in\mqn}q^{\lrmin(\sigma)}x^{\ap(\sigma)}y^{\bk_2(\sigma)}=\sum_{\tau\in \mqn^2}q^{\cyc(\tau)}x^{\caplat(\tau)}y^{\fix(\tau)},
\end{equation}
where $\bk_2(\sigma)$ is the number of blocks of $\sigma$ with length $2$.

The exponential generating function could be obtained as follows.
\begin{theorem}\label{thm_gen}
  \begin{equation*}
   \sum_{n\geq 0}\sum_{\tau\in \mqn^2}q^{\cyc(\tau)}x^{\caplat(\tau)}y^{\fix(\tau)}\frac{t^n}{n!}
   = \left(A\bigl(x,t\bigr)e^{t(y-1)}\right)^q,
  \end{equation*}
  where $A(x,z)$ is the exponential generating function given by~\eqref{Axz}.
\end{theorem}
\begin{proof}
In the same spirit as Theorem~\ref{thm02}, the generating function is equal to
\begin{equation*}
  \sum_{n\geq 0}\sum_{\pi\in \overline{\ms}_n} q^{\cyc(\pi)}x^{\asc(\pi)}y^{\fix(\pi)} \frac{t^n}{n!},
\end{equation*}
with the help of~\eqref{mqn2}.
The statistic $\fix(\pi)$ denotes the number of blocks of size 1 in $\pi$.

We use the language of combinatorial objects as in the book of Flajolet and Sedgewick~\cite{Flajolet2009}.
Take $\mathcal{C}$ to be the class of permutations in which the atom with label 1 is in the first place.
Then $\textsc{Set}(C)$ is exactly the class of permutations.
Let the generating function of $\mathcal{C}$ be
\begin{equation*}
  \Gen(\mathcal{C};\asc) := f(x,t) = t + x\frac{t^2}{2!} + (x+x^2)\frac{t^3}{3!} + O(t^4),
\end{equation*}
where the variable $x$ records the statistic ascent.
The interpretation of $\textsc{Set}(C)$ tells us that $e^{f(x,t)}=A(x,t)$.
Moreover,
\begin{equation*}
\sum_{n\geq 0}\sum_{\pi\in \overline{\ms}_n} q^{\cyc(\pi)}x^{\asc(\pi)} \frac{t^n}{n!} = e^{q\cdot f(x,t)}=A(x,t)^q
\end{equation*}
is also straightforward.
Now to take the statistic $\fix$ into account,
the first-order term of $f(x,t)$ should be changed from $t$ to $yt$. Hence the generating function in question is
\begin{equation*}
e^{q\cdot\left(f(x,t)-t+yt\right)} = \left(A\bigl(x,t\bigr)e^{t(y-1)}\right)^q
\end{equation*}
\end{proof}

A \emph{perfect matching} of $[2n]$ is a set partition of $[2n]$ with blocks (disjoint nonempty subsets) of size exactly 2.
Let $\mmn$ be the set of matchings of $[2n]$, and let $\m\in\mmn$.
The \emph{standard form} of $\m$ is a list of blocks $\{(i_1,j_1),(i_2,j_2),\ldots,(i_n,j_n)\}$ such that
$i_r<j_r$ for all $1\leq r\leq n$ and $1=i_1<i_2<\cdots <i_n$.
Throughout this paper we always write $\m$ in the standard form.
If $(i_s,j_s)$ is such a block of $\m$, then we say that $i_s$ (resp.~$j_s$) is an ascending (resp.~a descending) entry of $\m$.
\begin{definition}
Let $h_k$ be the number of
perfect matchings of $[4k]$ such that
$2i-1$ and $2i$ are either both ascending or both descending for
every $i\in[2k]$.
\end{definition}

Let $\textit{\i}^2={-1}$. Following~\cite{Sukumar07}, we have
\begin{equation*}
  \sqrt{\sec(2\textit{\i} z)}=\sum_{n\geq0}(-1)^nh_n\frac{z^{2n}}{(2n)!}.
\end{equation*}

\begin{definition}
A \emph{Stirling derangement} is a Stirling permutation without blocks of length $2$.
\end{definition}
Let $\mdq_n$ be the set of Stirling derangements of order $n$, i.e.,
$\mdq_n=\{\sigma\in\mqn|~\bk_2(\sigma)=0\}$.
The following result (in an equivalent form) has been algebraically obtained by Ma and Yeh~\cite[Page~15]{Ma1602}.
\begin{theorem}\label{thm04}
  \begin{equation*}
    \sum_{\sigma \in \mdq_n} (-1)^{\ap(\sigma)} =
    \left\{
      \begin{array}{ll}
        0, & \text{if $n$ is odd;} \\
        (-1)^k h_k, & \text{if $n = 2k$ is even.}
      \end{array}
    \right.
  \end{equation*}
\end{theorem}
In the rest of this section, we shall present a bijective proof of Theorem~\ref{thm04}.

Let $\varphi$ be the bijection between permutations and increasing binary trees defined by the inorder traversal in depth-first search.
The left-to-right minima of $\pi\in \msn$ will be the labels in the leftmost path of the corresponding trees.
Let $T_n$ denote the set of bicolored increasing binary trees on $n$
nodes such that all the nodes in the leftmost path are white and the other nodes are white or black.
Then $\varphi$ is a bijection from $\overline{\ms}_n$ to $T_n$ if we match white nodes (resp.~black nodes) with those unmarked letters (resp.~marked letters.)
Note that all of the left-to-right minima of a marked permutation are mapped to the
nodes of the leftmost path of its increasing binary tree. These
nodes will be called \emph{special nodes}.

For example, the Stirling permutation $\sigma = 223315544166$
corresponds to the marked permutation
$\pi = 2 \, 3 \, 1 \, \overline{5} \, \overline{4} \, 6$, which in
turn corresponds to the bicolored increasing tree with $6$ nodes in
Figure~\ref{fig:231546}. Note that the special nodes are labelled 1 and 2.
 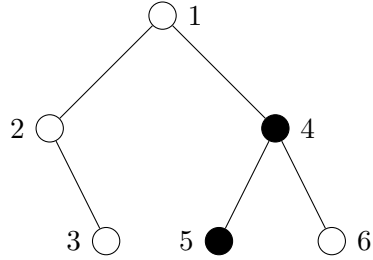
\begin{figure}[ht]
    \centering
    \begin{tikzpicture}
      \node (a) [circle, draw] {}
      child { node [circle, draw] {}
        child [missing]
        child { node [circle, draw] {} } }
      child [missing]
      child { node [circle, draw, fill] {}
        child { node [circle, draw, fill] {} }
        child { node [circle, draw] {} }
      };
      \node at (a) [right=0.2cm] {$1$};
      \node at (a-1) [left=0.2cm] {$2$};
      \node at (a-1-2) [left=0.2cm] {$3$};
      \node at (a-3) [right=0.2cm] {$4$};
      \node at (a-3-1) [left=0.2cm] {$5$};
      \node at (a-3-2) [right=0.2cm] {$6$};
    \end{tikzpicture}
    \caption{The bicolored tree corresponded to $2 \, 3 \, 1 \,
      \overline{5} \, \overline{4} \, 6$}
    \label{fig:231546}
  \end{figure}

\medskip

\noindent{\bf A bijective proof of Theorem~\ref{thm04}.}\\
Now we consider the composition $\varphi \circ \Phi$ from
$\mqn$ to $T_n$, where $\Phi$ is defined in the proof of Theorem~\ref{thm01}.
It is easily observed that $k$ is an ascent plateau of $\sigma \in \mqn$ if and only if the node $k$ in the bicolored increasing binary tree $\varphi(\theta(\sigma))$ is non-special and does not have a left child.
Also the Stirling derangements are mapped to those trees each of whose special nodes has a right child.

For trees in $T_n$, there are some trees that
have a non-special node with only one child.  An involution can be introduced on
these trees: among the non-special nodes with only one child, find
the one with the minimal label, then move its whole subtree to the
other branch.  An example on this involution is shown in
Figure~\ref{fig:inv-tree}.

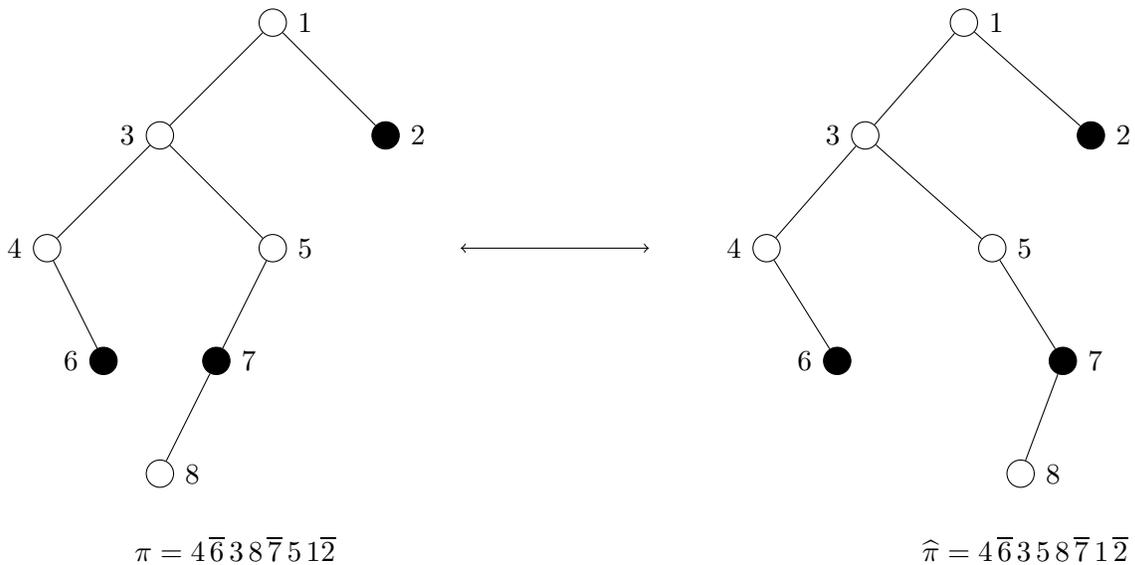
\begin{figure}[ht]
  \centering
  \begin{tikzpicture}
    \node (a) [circle, draw] {}
      child { node [circle, draw] {}
        child { node [circle, draw] {}
          child [missing]
          child { node [circle, draw, fill] {}
          }}
        child [missing]
        child { node [circle, draw] {}
          child { node [circle, draw, fill] {}
            child { node [circle, draw] {} }
            child [missing] }
          child [missing] }
        }
      child [missing]
      child { node [circle, draw, fill] {} };
    \node at (a) [right=0.2cm] {$1$};
    \node at (a-1) [left=0.2cm] {$3$};
    \node at (a-1-1) [left=0.2cm] {$4$};
    \node at (a-1-1-2) [left=0.2cm] {$6$};
    \node at (a-1-3) [right=0.2cm] {$5$};
    \node at (a-1-3-1) [right=0.2cm] {$7$};
    \node at (a-1-3-1-1) [right=0.2cm] {$8$};
    \node at (a-3) [right=0.2cm] {$2$};
    \node at (-0.5,-7) {$\pi = 4 \, \overline{6} \, 3 \, 8 \,
      \overline{7} \, 5 \, 1 \overline{2}$};
    \draw [<->] (2.5,-3) -- (5,-3);
    \begin{scope}[right=9cm]
      \node (a) [circle, draw] {}
        child { node [circle, draw] {}
          child { node [circle, draw] {}
            child [missing]
            child { node [circle, draw, fill] {}
            }}
            child [missing]
            child { node [circle, draw] {}
              child [missing]
              child { node [circle, draw, fill] {}
                child { node [circle, draw] {} }
                child [missing] }
             }
          }
        child [missing]
        child { node [circle, draw, fill] {} };
      \node at (a) [right=0.2cm] {$1$};
      \node at (a-1) [left=0.2cm] {$3$};
      \node at (a-1-1) [left=0.2cm] {$4$};
      \node at (a-1-1-2) [left=0.2cm] {$6$};
      \node at (a-1-3) [right=0.2cm] {$5$};
      \node at (a-1-3-2) [right=0.2cm] {$7$};
      \node at (a-1-3-2-1) [right=0.2cm] {$8$};
      \node at (a-3) [right=0.2cm] {$2$};
      \node at (-0.5,-7) {$\widehat{\pi} = 4 \, \overline{6} \, 3 \, 5
        \, 8 \, \overline{7} \, 1 \, \overline{2}$};
    \end{scope}
  \end{tikzpicture}
  \caption{Involution on trees in $T_n$; the
    move is made on Node 5.}
  \label{fig:inv-tree}
\end{figure}

Following the discussion in the previous
paragraph, we immediately see that this involution is sign-reversing
if we attach the sign $(-1)^{\ap(\sigma)}$ to the tree
$\varphi(\theta(\sigma))$.

For those trees to which the above involution cannot be applied, all
of its nodes have either zero or two children; these trees are
called \emph{complete}, whose number of nodes $n$ must be even, and
it has the sign $(-1)^{\ap(\sigma)}=(-1)^{n/2}$. From here we
conclude that when $n$ is \emph{odd}, these trees are
sign-balanced.  Our remaining task is to give a proper count on
those complete bicolored increasing binary trees on $n = 2k$ nodes.

Let $\sigma \in \mdq_n$ such that $\varphi(\theta(\sigma))$ is
complete. Then the intermediate marked permutation
\begin{equation*}
  \pi = \theta(\sigma) = \pi_1 \pi_2 \dots \pi_n
\end{equation*}
is reverse alternating, i.e., $\pi_1<\pi_2> \pi_3< \cdots$.
Note that in these situations the left-to-right minima of $\pi$
always occur at odd-indexed positions. We now apply the
Foata transformation on $\pi$ by making its left-to-right minima as
the head of each cycle; say $\pi \mapsto C_1 C_2 \dots C_j$.  Each
cycle determines some pairings on $\pm[n]$: for any
$C = ( c_1 c_2 \dots c_\ell )$, we pair $\overline{c_i}$ with $c_{i+1}$ for $1 \leq i < \ell$, and pair $\overline{c_\ell}$ with
$c_1$.  Since the number at the beginning of a cycle must be unmarked,
we have a bijection between complete bicolored increasing
binary trees on $n$ nodes with those special perfect matchings on $[4k] = [2n] \simeq \pm[n]$. Using
definition of the numbers $h_k$, we get the desired result. \qed

\begin{ex} We use the ordering
$\overline{1} < 1 < \overline{2} < 2 < \cdots < \overline{n} < n$ on
$\pm[n] \cong [2n] = [4k]$.  The reverse alternating marked
permutation $\pi = 3 \, \overline{5} \, 1 \, \overline{6} \, \overline{2} \, 4$
corresponds to the following perfect matching on $\pm[6]$:
\begin{equation*}
  \overline{3} \, \overline{5} /
  5 \, 3 /
  \overline{1} \, \overline{6} /
  6 \, \overline{2} /
  2 \, 4 /
  \overline{4} \, 1.
\end{equation*}
Because $\pi$ is reverse alternating, the numbers $i$ and $\overline{\imath}$ are both
ascending or both descending for all $i \in [n]$ in the corresponding
perfect matching.
\end{ex}

\section{On $r$-colored marked permutations}
\label{Section04}

The definition of marked permutations can easily be extended to the
$r$-colored version for any positive integer $r$.
We think that a marked permutation (resp. an ordinary permutation) is in the special case $r=2$ (resp. $r=1$) in which the unmarked elements are painted by the first color, while the others are painted by the second color.
The definition for the \emph{$r$-colored marked permutations} is similar to that in the two-colored situation, i.e., every left-to-right minimum of $\pi$ must be painted by the first color, which is always referred as white in this section.
We find that on the $r$-colored marked permutations the ascent statistic is equidistributed with another statistic, which is defined below.

Let $\overline{\ms}_n^{(r)}$ be the set of $r$-colored marked permutations of $[n]$.
For $\pi\in \overline{\ms}_n^{(r)}$,
we define
  \begin{equation*}
    \desrlmin(\pi) = \Bigl|
      \{ \pi_i \mid \text{$\pi_{i-1} > \pi_i$ for $i > 1$, or $\pi_i$ is a right-to-left minimum of $\pi$} \}
    \Bigr|.
  \end{equation*}

\begin{theorem}\label{thm03}
  Let $r$ be a positive integer. The following polynomial is symmetric in the variables $x$ and
  $y$:
  \begin{equation*}
    F_r(x,y) = \sum_{\pi \in \overline{\ms}_n^{(r)}}
    x^{\asc(\pi)+1}y^{\desrlmin(\pi)}.
  \end{equation*}
\end{theorem}

\begin{proof}
We will construct an involution $\psi$ on $\ms_n$ first.  Let
$\pi = \pi_1 \pi_2 \dots \pi_n \in \ms_n$. Again we divide $\pi$
into several blocks by marking each right-to-left minimum of $\pi$ as
the end of a block; we will view each block as a cycle.  Now we
rotate each block in order to list its largest number at the end of
its block; then arrange these blocks so that these largest numbers
are in decreasing order; lastly take the complement of every number
to get the final result $\psi(\pi)$.

For example, let us go through step by step on the permutation $\pi= 214853697 \in \ms_9$:
  \begin{equation*}
    \begin{split}
      \pi = 21\textbf{48}53\textbf{69}7 & \to 21 | 4853 | 6 | 97 \\
      & \to 12 | 5348 | 6 | 79 \\
      & \to 79 | 5348 | 6 | 12 \\
      & \to 31 | 5762 | 4 | 98 \mapsto 3\textbf{1}57\textbf{624}9\textbf{8} = \psi(\pi).
    \end{split}
  \end{equation*}
A careful examination shows that $\psi$ is an involution on $\ms_n$, and $\asc(\pi)+1 = \desrlmin(\psi(\pi))$.

For $r > 1$, we can take a certain algorithm to ensure that every
number in $\psi(\pi)$ is still colored white. For example,
consider the colors are taken modulo $r$ by addition in each block;
or during the block rotation process we keep the colors fixed at
their respective positions.
Since the above involution $\psi$ does not depend on the
coloring, we still have the result in symmetry.
\end{proof}

The generating function
$$F_r(x,y) := \sum_{\pi \in \overline{\ms}_n^{(r)}} x^{\asc(\pi)+1}y^{\desrlmin(\pi)}$$ can be obtained by using context-free grammars,
as done by Dumont~\cite{Dumont1996}.
For $r = 1$, the grammar is
$\{ a \to ab, b \to ab, c \to ab, e \to ce \}$, which had been studied
by Roselle~\cite{Roselle1968}.
Let $A(x,y,z,t)$ be a refinement for the Eulerian polynomials:
\begin{equation*}
  A(x,y,z,t) = zt + \sum_{n \geq 2} \sum_{\pi \in \ms_n}
  x^{\asc(\pi)+1} y^{\des(\pi)} \frac{t^n}{n!}.
\end{equation*}
Then the generating function for $F_1(x,y)$ is shown
in~\cite{Dumont1996} to be
\begin{equation*}
  G_1(x,y,t) := \sum_{n \geq 0} F_1(x,y) \frac{t^n}{n!} = \exp \bigl( A(x,y,xy,t)
  \bigr).
\end{equation*}
We define $\displaystyle G_r(x,y,t) := \sum_{n \geq 0} F_r(x,y) \frac{t^n}{n!}$.
For any $r$-colored permutation, it can also be partitioned into blocks by the method described in the proof of Theorem~\ref{thm03}. Then we can assign those blocks whose end entries are painted by the same color into the same group. It is clear that the number of $r$-colored marked permutations of length $n$ is the same as those $r$-colored permutations of length $n$ whose right-to-left minima are all painted by any other color.  Hence we have the identity:
\begin{equation*}
  G_1(x,y,rt) = \bigl( G_r(x,y,t) \bigr)^r,
\end{equation*}
that is,
\begin{equation*}\label{eq:r-mark}
  G_r(x,y,t) = \bigl( G_1(x, y, rt) \bigr)^{1/r}.
\end{equation*}

\section{An interpretation of $r$-colored marked permutations}\label{Section05}
Suppose a class of combinatorial objects $\C$ has the generating function $\Gen(\C)=f(\textbf{x},t)$, where $\textbf{x}$ is the vector $\textbf{x}=(x_1,x_2,\dots,x_s)$, and the power of $x_i$ records certain statistic $\stati$ on $\C$.
Here the generating function can be either ordinary or exponential, as an ordinary one can be
seen as an exponential one with arbitrary labelling.

Recall that
\begin{equation*}
g(\textbf{x}, t,r_1,r_2,\dots,r_k) := f(\textbf{x},(r_1+r_2+\dots+r_k)t)
\end{equation*}
counts a weighted version of $\mathcal{C}$, denoted by ${}^k\mathcal{C}$, where
each atom in an object in $\mathcal{C}$ is weighted by a number in $[k]$.
An object in ${}^k\mathcal{C}$ with $n$ atoms could be seen as $(C,w)$, where
$C\in\mathcal{C}$ and $w$ is a map from $[n]$ to $[k]$.
Then the power of $x_i$ in $g$ records $\stati(C)$, and the power of $r_j$ records the number of atoms in $C$ with weight $j$, i.e., the cardinality of the pre-image $w^{-1}(j)$.

It is well-known that for a given non-negative integer $k$,
\begin{equation*}
  f_{{\rm seq},k}(\textbf{x}, t) := f(\textbf{x},t)^k
\end{equation*}
counts $\textsc{Seq}_k(\mathcal{C})$. Each object in $\textsc{Seq}_k(\mathcal{C})$ is of the form of a $k$-sequence: $(C_1,C_2,\dots,C_k)$ with each $C_i\in\mathcal{C}$.
The power of $x_i$ in $h_s$ records $\stati(C_1)+\stati(C_2)+\dots+\stati(C_k)$.
On the other hand,
\begin{equation*}
  f_{{\rm set},k}(\textbf{x}, t) := \frac{f(\textbf{x},t)^k}{k!}
\end{equation*}
counts $\textsc{Set}_k(\mathcal{C})$.  This induced class contains objects of the form of a $k$-set $\{C_1,C_2,\dots,C_k\}$ with each $C_i\in\mathcal{C}$.
The power of $x_i$ in $h_s$ also records $\stati(C_1)+\stati(C_2)+\dots+\stati(C_k)$.

When $\mathcal{C}$ has no objects of size 0, the power series
\begin{equation*}
  f_{{\rm seq}}(\textbf{x}, t) := \frac{1}{1-f(\textbf{x},t)}
\end{equation*}
is well-defined, and it counts $\textsc{Seq}(\mathcal{C}):=\cup_{k=0}^\infty \textsc{Seq}_k(\mathcal{C})$; while the power series
\begin{equation*}
  f_{{\rm set}}(\textbf{x}, t) := \exp \left({f(\textbf{x},t)}\right)
\end{equation*}
counts $\textsc{Set}(\mathcal{C}):=\cup_{k=0}^\infty \textsc{Set}_k(\mathcal{C})$.

Suppose now that $\mathcal{C}$ is a combinatorial objects with
no objects of size 0, with generating function $\Gen(\mathcal{C})=f(\textbf{x},t)$.
Observe the formal power series
\begin{equation*}
  h(\textbf{x},t):=\frac{1}{k}f(\textbf{x},kt)
\end{equation*}
has non-negative integer coefficients.
This fact suggests that there should be some general method to induce a
new class ${}^{*k}\mathcal{C}$ of combinatorial objects that is enumerated by $h$.
Suppose we have a rule to specify an atom from any object $C$ in $\mathcal{C}$; that atom will be called a \emph{distinguished atom} of $C$. For example, let the distinguished atom be the one with label 1.
Then we could define ${}^{*k}\mathcal{C}$ to be a weighted version
of $\mathcal{C}$, where each atom except the distinguished one in an object in $\mathcal{C}$ is weighted by a number in $[k]$.
The combinatorial interpretation of the refined generating function
\begin{equation*}
  h(\textbf{x},t,r_1,r_2,\dots,r_k):=\frac{1}{r_1+r_2+\dots+r_k}f(\textbf{x},(r_1+r_2+\dots+r_k)t)
\end{equation*}
is straightforward.

Take $\mathcal{C}$ to be the class of permutations in which the atom with label 1 is in the first place.
Then $\textsc{Set}(C)$ (resp.~$\textsc{Set}({}^{*r}\mathcal C)$) is exactly the class of permutations (resp.~$r$-colored marked permutations) that we introduced in this article.


\end{document}